\documentclass[10pt,reqno]{amsart}

\usepackage{amssymb}
\usepackage{xypic}
\xyoption{all}
\def\umono{\ar@{_{(}->}[u]}
\def\uumono{\ar@{_{(}->}[uu]}

\def\lmono{\ar@{_{(}->}[l]}
\def\llmono{\ar@{_{(}->}[ll]}

\newcommand{\Z}{{\mathbb Z}}

\newcommand{\C}{{\mathbb C}}
\newcommand{\F}{{\mathbb F}}

\newcommand{\pcom}{_{p}^{\wedge}}

 \renewcommand{\P}[1]{\mathcal{P}^{#1}}

\newcommand{\map}{\operatorname{map}\nolimits}

                     \newcommand{\Hom}{\operatorname{Hom}\nolimits}

\newcommand{\Aut}{\operatorname{Aut}\nolimits}
\newcommand{\aut}{\operatorname{aut}\nolimits}

\renewcommand{\ker}{\operatorname{Ker}\nolimits}

\newcommand{\A}{\ifmmode{\mathcal{A}}\else${\mathcal{A}}$\fi}
\newcommand{\K}{\ifmmode{\mathcal{K}}\else${\mathcal{K}}$\fi}
\newcommand{\U}{\ifmmode{\mathcal{U}}\else${\mathcal{U}}$\fi}
\newcommand{\T}{\ifmmode{\mathcal{T}}\else${\mathcal{T}}$\fi}
\newcommand{\FF}{\ifmmode{\mathcal{F}}\else${\mathcal{F}}$\fi}
\newcommand{\LL}{\ifmmode{\mathcal{L}}\else${\mathcal{L}}$\fi}

\newtheorem{theorem}{Theorem}[section]
\newtheorem{proposition}[theorem]{Proposition}
\newtheorem{corollary}[theorem]{Corollary}
\newtheorem{lemma}[theorem]{Lemma}
\newtheorem{definition}[theorem]{Definition}
\newtheorem{remark}[theorem]{Remark}
\newtheorem{example}[theorem]{Example}

\title[Noetherian loop spaces] {Noetherian loop spaces}

\author{Nat\`{a}lia Castellana}
\author{Juan A. Crespo}
\author{J\'er\^{o}me Scherer${}^*$}

\thanks{*: corresponding author. The first and third  authors are partially supported by FEDER/MEC grant
MTM2007-61545. The second author is partially supported by FEDER/MEC
grant SEJ2007-67135. The third author has benefitted from the LMS
scheme 2 grant number 2618.}

\subjclass[2000]{Primary 55P45; Secondary 13D03, 55S10, 55T20}

\begin{document}


\begin{abstract}
The class of loop spaces whose mod $p$ cohomology is Noetherian is
much larger than the class of $p$-compact groups (for which the mod
$p$ cohomology is required to be finite). It contains Eilenberg-Mac
Lane spaces such as $\C P^\infty$ and $3$-connected covers of
compact Lie groups. We study the cohomology of the classifying space
$BX$ of such an object and prove it is as small as expected, that
is, comparable to that of $B\C P^\infty$. We also show that $BX$
differs basically from the classifying space of a $p$-compact group
in a single homotopy group. This applies in particular to
$4$-connected covers of classifying spaces of Lie groups and sheds
new light on how the cohomology of such an object looks like.
\end{abstract}


\maketitle


\section*{Introduction}
\label{sec intro}
From the point of view of homotopy theory compact Lie groups are
finite loop spaces, i.e. triples $(X, BX, e)$ where $X$ is a finite
complex, and $e: X \rightarrow \Omega BX$ is a homotopy equivalence.
Most of their geometric features are captured $p$-locally in
homotopy theory, where $p$ is any prime, as shown by Dwyer and
Wilkerson in~\cite{MR1274096}. They introduced the notion of
$p$-compact group, replacing the finiteness condition by a
cohomological one, namely that $H^*(X; \F_p)$ must be finite, and
requiring the additional property that $BX$ be local with respect to
mod $p$ homology, or equivalently $p$-complete, \cite{MR51:1825}. It
is the ``classifying space" $BX$ which carries all of the
information about the loop space. Amazingly enough, apart from
compact Lie groups, there are only a few families of exotic
$p$-compact groups and they have been recently completely classified
by Andersen, Grodal, M{\o}ller, and Viruel, see
\cite{classificationp} for the odd prime case and
\cite{classification2}, \cite{MR2314074}, \cite{MR2338539} for the
prime~$2$ (the only exotic $2$-compact group is basically the space
$DI(4)$ constructed by Dwyer and Wilkerson, \cite{MR1161306}).

If one aims at an understanding not only of finite loop spaces, but
larger ones, the next natural step to take is to relax the
cohomological finiteness condition. We define thus a
\emph{$p$-Noetherian group} to be a loop space $(X, BX, e)$ where
$BX$ is $p$-complete and $H^*(X; \F_p)$ is a finitely generated
(Noetherian) $\F_p$-algebra. Relying on Bousfield localization
techniques, \cite{B2}, and Miller's solution to the Sullivan
conjecture, \cite{Miller}, more precisely on Lannes' $T$-functor
technology, \cite{MR93j:55019}, we describe the structure of
$p$-Noetherian groups and their relation to $p$-compact groups. We
compute qualitatively the cohomology of the classifying space~$BX$
and obtain new general results when $X$ is the $3$-connected cover
of a compact Lie group, see Corollary~\ref{cor 4connectedcover}.

Let us be more precise. In the case of $p$-compact groups, i.e.
when the mod $p$ cohomology of the loop space $H^*(X; \F_p)$ is
finite, Dwyer and Wilkerson's main theorem in \cite{MR1274096}
shows that there are severe restrictions on the cohomology of the
classifying space: $H^*(BX;\F_p)$ is always a finitely generated
(Noetherian) $\F_p$-algebra. Likewise the cohomology of the
classifying space of a $p$-Noetherian group cannot be arbitrarily
large.

\medskip
\noindent{\bf Theorem~\ref{theorem BX and Ap}}
{\em Let $(X, BX, e)$ be a $p$-Noetherian group. Then $H^*(BX;
\F_p)$ is finitely generated as an algebra over the Steenrod
algebra.}
\medskip

This information is not optimal as it does not permit to decide for
example when an Eilenberg-Mac Lane space of type $K(\Z/p, m)$ is a
$p$-Noetherian group. The mod $p$ cohomology of any of them is
finitely generated as an algebra over~$\A_p$, but the only
classifying spaces of a $p$-Noetherian group are $K(\Z/p, 1)$ and
$K(\Z/p, 2)$. Schwartz's Krull filtration of the category $\U$ of
unstable modules is an established and convenient tool to measure
how large an unstable algebra is, \cite{MR95d:55017}. The Krull
filtration $\U_0\subset\U_1\subset\dots$ is defined inductively,
starting with the full subcategory $\U_0$ of $\U$ of locally finite
unstable modules (the span of every element under the action of the
Steenrod algebra is finite). In fact the cohomology $H^*(X; \F_p)$
is locally finite if and only if the space $X$ is $B\Z/p$-local,
i.e. the evaluation $\map(B\Z/p, X) \rightarrow X$ is a weak
equivalence (\cite[Th\'eor\`eme 0.14]{MR827370}).

There are many $B\Z/p$-local spaces, but there are none for which
the cohomology lies in higher stages of the Krull filtration. This
is the statement of Kuhn's non-realizability
conjecture~\cite{MR96i:55027}, which has been settled by Schwartz
in~\cite{MR99j:55019} and~\cite{MR2002k:55043}, and proved in its
full generality by Dehon and Gaudens in~\cite{MR1997324}. Thus the
cohomology of a space lies in $\U_0$ or it does not lie in any
$\U_n$. The cohomology $H^*(K(\Z/p, m); \F_p)$ for example does
not lie in any $\U_n$, but the quotient module of indecomposable
elements $QH^*(K(\Z/p, m); \F_p)$ lies in $\U_{m-1}$ for any $m
\geq 1$ (\cite[Example 2.2]{MR2264802}).

It was observed in \cite[Lemma 7.1]{MR2264802} that if $H^*(BX;
\F_p)$ is finitely generated as an algebra over~$\A_p$, then
$QH^*(BX; \F_p)$ must be finitely generated as a module
over~$\A_p$, and hence lies in $\U_k$ for some~$k$. We remark that
the condition that $H^*(X; \F_p)$ be a Noetherian $\F_p$-algebra
is equivalent to saying that $H^*(X;\F_p)$ is finitely generated
as an algebra over~$\A_p$ and that the unstable module
$QH^*(X;\F_p)$ of indecomposable elements lies in $\U _0$. Our
second result shows that the cohomology of the classifying space
of a $p$-Noetherian group is as small as expected in terms of the
Krull filtration.

\medskip

\noindent{\bf Theorem~\ref{thm TQforpNoetherian}}
{\em Let $(X, BX, e)$ be a $p$-Noetherian group. Then  $QH^*(BX;
\F_p)$ belongs to $\U_1$.}
\medskip

The arguments to prove these results are the following. We start
in Section~\ref{sec pNoetherian} with the study of the structure
of $p$-Noetherian groups. The most basic examples of
$p$-Noetherian groups $X$ are $p$-compact groups and Eilenberg-Mac
Lane spaces $K(\Z/p^r, 1)$ and, $ K(\Z^\wedge_p,2)$. In the spirit
of our deconstruction results for $H$-spaces, \cite{MR2264802}, we
show that these are the basic building blocks for all
$p$-Noetherian groups.

\medskip

\noindent{\bf Theorem~\ref{structure of BX}}
{\em Let $(X, BX, e)$ be any $p$-Noetherian group. There exists then
a fibration
$$
K(P, 2)^\wedge_p \longrightarrow BX \longrightarrow BY\, ,
$$
where $P$ is a $p$-group which is a finite direct sum of copies of
cyclic groups and Pr\"ufer groups and $Y$ is a $p$-compact group.
}
\medskip

We stress that this fibration is functorial in $BX$ as the base
space is obtained by Bousfield localization. The understanding of
the cohomology of $BX$ goes through the analysis of the Serre
spectral sequence of this fibration. Note that, by Dwyer and
Wilkerson's main theorem in~\cite{MR1274096}, $H^*(BY; \F_p)$ is
finitely generated as an algebra. Also, the mod $p$ cohomology of
$K(P, 2)^\wedge_p$ is finitely generated as an algebra over the
Steenrod algebra by \cite{MR0060234} and \cite{MR0087934}. The
spectral sequence is not nearly as nice as what happens for
$H$-spaces though, compare with \cite{CCS5}, and we must first
tackle Theorem~\ref{thm TQforpNoetherian}.

This we do in Sections~\ref{sec splitting} -- \ref{sec QandU} by
giving first a geometric interpretation to $\bar T QH^*(BX;
\F_p)$, where $\bar T$ is Lannes' reduced $T$-functor. Recall that
for ``nice" spaces (such as $BX$) the unreduced $T$ functor $T
H^*(BX; \F_p)$ computes the cohomology of the mapping space $\map
(B\Z/p, BX)$ and we rely on Schwartz's characterization
\cite[Theorem~6.2.4]{MR95d:55017} of the Krull filtration in terms
of~$\bar T$. In order to perform our reduced $T$ functor
calculation, we prove that the component $\map (B\Z/p, BX)_c$ of
the constant map splits as a product $BX \times \map_* (B\Z/p,
BX)_c$. Using the properties of the $T$ functor, this splitting
yields a geometric interpretation of the reduced $T$ functor in
terms of the pointed mapping space, more precisely $\bar T QH^*(X;
\F_p) \cong H^*(\map_*(B\Z/p, BX)_c; \F_p)$.

We finally come back to Theorem~\ref{theorem BX and Ap}, which we
prove in two steps. First we investigate in Section~\ref{sec
Serre} the Serre Spectral sequence for fibrations over spaces with
finite cohomology and fibre a finite product of Eilenberg
Mac--Lane spaces. We show that in this situation the cohomology of
the total space is finitely generated as an algebra over~$\A_p$.
Secondly, we reduce the situation in which the base space of the
fibration is a $p$-compact toral group.

\medskip

Let us conclude the introduction with two remarks.
The results in this article show that for a compact simply-connected
Lie group $G$, the module $QH^*((BG)\langle 4 \rangle; \F_p)$ is
finitely generated over~$\A_p$ and belongs to $\U_1$. This puts into
context the calculations made by Harada and Kono, \cite{MR896947},
and sheds new light on how the cohomology will look like even in the
cases where an explicit description has not been obtained, compare
with Example~\ref{more than HK}.

Let us finally mention Corollary~\ref{cor finitecover} in which we
obtain, as a byproduct of our theory, a description of the
cohomology of the $n$-connected cover of an $(n-1)$-connected finite
complex. This was only known for $H$-spaces before, \cite{CCS5}.

\medskip

{\bf Acknowledgments.} This project originated at the
Mittag-Leffler Institute during the emphasis semester on homotopy
theory in 2006. We would like to thank Kasper Andersen, Jesper
Grodal, Frank Neumann, and Alain Jeanneret for their interest and
Akira Kono for pointing out a problem in an earlier version of
this article.

\section{The structure of $p$-Noetherian groups}
\label{sec pNoetherian}
This first section is devoted to the description of the
classifying space of $p$-Noetherian groups and the relation to
$p$-compact groups. Let us start with the definition and the basic
examples. The statements about the action of the Steenrod algebra
and the Krull filtration will be explained and developed in
Section~\ref{sec QandU}.

\begin{definition}
\label{def pNoetherian}
{\rm A \emph{$p$-Noetherian group} is a triple $(X, BX, e)$ where
$H^*(X; \F_p)$ is a Noetherian algebra (it is finitely generated as
an algebra), $BX$ is a $p$-complete space, and $e: X \rightarrow
\Omega BX$ is a weak equivalence.}
\end{definition}

We will often use the word $p$-Noetherian group for the loop space
$X$ and refer to $BX$ as the classifying space. Hence, we will say
that a $p$-Noetherian group is $n$-connected if so is the loop space
$X$, or equivalently if the classifying space is $(n+1)$-connected.
Note that if the \emph{integral} cohomology of a space is
Noetherian, as an algebra, then so is the mod $p$ cohomology (just
like finite loop spaces have finite mod $p$ cohomology).

\begin{remark}
\label{p-good}
{\rm Since $\pi_1(BX)\cong \pi_0(X)$ and $H^*(X;\F_p)$ is of
finite type, it follows that $\pi_1(BX)$ is finite and, therefore
$BX$ is a $p$-good space by \cite[VII, Proposition
5.1]{MR51:1825}. In fact, $\pi_1(BX)$ is a $p$-group by \cite[VII,
Proposition 4.3]{MR51:1825}.}
\end{remark}

\begin{example}
\label{ex pcg}
{\rm In \cite{MR1274096}, Dwyer and Wilkerson introduced the
notion of a $p$-compact group. A $p$-compact group is a loop space
$(X,BX,e)$ such that $BX$ is $p$-complete and $H^*(X;\F_p)$ is a
finite $\F_p$-vector space. It is clear from the definition that
$p$-compact groups are $p$-Noetherian groups. }
\end{example}

The most basic example of $p$-compact group is given by the
$p$-completed circle and its classifying space $K(\Z^\wedge_p, 2)$.
Our definition of $p$-Noetherian group allows to include not only all
$p$-compact groups but also the following Eilenberg-Mac Lane
spaces.

\begin{example}
\label{ex K(Z,2)}
{\rm Let $X = K(\Z^\wedge_p, 2)$, $BX = K(\Z^\wedge_p, 3)$, and $e$
the obvious homotopy equivalence between $\Omega K(\Z^\wedge_p, 3)$
and $K(\Z^\wedge_p, 2)$. This is a $p$-Noetherian group since
$H^*(K(\Z^\wedge_p, 2); \F_p) \cong \F_p[u]$ is finitely generated
as an algebra. Let us point out here that $H^*(K(\Z^\wedge_p,
3);\F_p)$ is finitely generated as an algebra over $\A_p$ and that
the module of indecomposable elements $QH^*(K(\Z^\wedge_p, 3);
\F_p)$ lives in $\U_1$. For example $QH^*(K(\Z^\wedge_2, 3); \F_2)
\cong \Sigma F(1)$, where $F(1)$ is the free unstable module on one
generator in degree~$1$.}
\end{example}

In fact, this is basically the only $1$-connected $p$-Noetherian
group $X$ such that $\Omega X$ is $\F_p$-finite.

\begin{proposition}
\label{prop torus}
Let $(X, BX, e)$ be a $p$-Noetherian group such that $H^*(\Omega X;
\F_p)$ is finite. Then $BX$ is 2-connected if and only if it is a
product of a finite number of copies of $K(\Z^\wedge_p, 3)$.
\end{proposition}

\begin{proof}
The loop space $\Omega X\simeq \Omega^2 BX$ is a connected homotopy
commutative mod $p$ finite $H$-space. Thus, by the mod $p$ version
of Hubbuck's Torus Theorem, \cite{MR38:6592} and \cite{Aguade}, we
see that $\Omega X$ is equivalent to a $p$-completed torus.
\end{proof}

\begin{example}
\label{ex BS33}
{\rm Let us consider the compact Lie group $S^3$ and its
$3$-connected cover $S^3 \langle 3 \rangle$. Identifying
$B(S^3\langle 3\rangle)$ with $(BS^3)\langle 4\rangle$, we have a
fibration $K(\Z, 3) \rightarrow (BS^3)\langle 4 \rangle
\rightarrow BS^3$. The triple $(S^3 \langle 3 \rangle^\wedge_p,
B(S^3) \langle 4 \rangle^\wedge_p, e)$ is then a $p$-Noetherian
group since $H^*(S^3 \langle 3 \rangle; \F_p) \cong \F_p[x]
\otimes E(y)$ where $x$ has degree $2p$, $y$ has degree $2p+1$,
and a Bockstein connects $x$ and $y$, $\beta(x)=y$. This
$p$-Noetherian group is an extension of the $p$-compact group
$(S^3)^\wedge_p$ and the Eilenberg-Mac Lane space from
Example~\ref{ex K(Z,2)}.

It is not difficult to compute the mod $p$ cohomology of
$(BS^3)\langle 4\rangle$ using the Serre spectral sequence. For
example, $H^*(B(S^3 \langle 3 \rangle); \F_2) \cong \F_2[z, Sq^1
z, Sq^4 z, Sq^{8,4} z, \dots]$, where $z$ has degree $5$. This is
a subalgebra of the mod $2$ cohomology of $K(\Z, 3)$ and $z$
corresponds to $Sq^2 \iota_3$, where $\iota_3$ is the fundamental
class. Again, we see that the module of indecomposable elements
belongs to $\U_1$, as it differs from $\Sigma F(1)$ by only a few
classes in low degrees.}
\end{example}

This last example fits into a more general picture. One can
consider $4$-connected covers of classifying spaces of compact Lie
groups.

\begin{example}
\label{Harada-Kono}
{\rm Let $G$ be a simply connected compact Lie group and consider
the $4$-connected cover of its classifying space, $(BG) \langle 4
\rangle$. Since the mod $p$ cohomology $G\langle 3 \rangle$ is
Noetherian, this provides an infinite number of examples of
$p$-Noetherian groups.

Harada and Kono studied in \cite{MR896947} and \cite{MR750280} the
fibration $K(\Z, 3) \rightarrow (BG) \langle 4 \rangle \rightarrow
BG$. They were able to compute explicitly, as an algebra, the
cohomology of the total space at odd primes and a few cases at the
prime~$2$. In all of these computations the result is the tensor
product of a quotient of $H^* (BG; \F_p)$ with a certain
subalgebra of $H^* (K(\Z, 3); \F_p)$, which turns out to be always
finitely generated as an algebra over the Steenrod algebra.}
\end{example}

In fact, $p$-Noetherian groups are closed under fibrations in the
following sense (and this explains why $G \langle 3 \rangle$
defines a $p$-Noetherian group).

\begin{proposition}
\label{closed-under-fibrations}
Let $BX\rightarrow E \rightarrow BZ$ be a fibration of connected
spaces where $BX$ and $BZ$ are classifying spaces of
$p$-Noetherian groups. Then $E$ is also the classifying space of a
$p$-Noetherian group.
\end{proposition}

\begin{proof}
Since $\pi_1(BZ)$ is a finite $p$-group and both $BX$ and $BZ$ are
$p$-complete and $p$-good (see Remark \ref{p-good}), the fibre
lemma \cite[II.5.1]{MR51:1825} shows that $E$ is also
$p$-complete. It remains to show that $H^*(\Omega E; \F_p)$ is a
finitely generated algebra.

Looping the fibration, we obtain an $H$-fibration $X\rightarrow
\Omega E \rightarrow Z$ where both $X$ and $Z$ have finitely
generated mod $p$ cohomology. By \cite[Theorem~5.1]{CCS5}, the
cohomology of $\Omega E$ is finitely generated as an algebra over
the Steenrod algebra, in other words $QH^*(\Omega E; \F_p)$ is
finitely generated as an $\A_p$-module. This unstable module is
thus finite if and only if it is locally finite, which by
\cite{MR92b:55004} is equivalent to the evaluation
$\map(B\Z/p,\Omega E)_c\rightarrow E$ to be an equivalence. This
follows from the fact that this is the case for $X$ and $Z$.
\end{proof}

Let us analyze the structure of an arbitrary connected
$p$-Noetherian group. The following theorem  tells us that it
always differs from a $p$-compact group in a single $p$-completed
Eilenberg-MacLane space.

\begin{theorem}
\label{structure of BX}
Let $(X, BX, e)$ be any $p$-Noetherian group. There exists then a
fibration
$$
K(P, 2)^\wedge_p \rightarrow BX \rightarrow BY\, ,
$$
where $P$ is a finite direct sum of copies of cyclic groups and
Pr\"ufer groups and $Y$ is a $p$-compact group.
\end{theorem}

\begin{proof}
By assumption the mod $p$ cohomology of $X$ is finitely generated
as an algebra. In other words, the module of indecomposable
elements $QH^*(X; \F_p)$ is finite. Therefore, by
\cite[Theorem~3.2]{MR92b:55004} the loop space $\Omega X$ is
$B\Z/p$-local, or equivalently the classifying space $BX$ is
$\Sigma^2B\Z/p$-local (\cite[Theorem~3.A.1]{Dror}). The analysis
in \cite{B2} by Bousfield of the Postnikov like nullification
tower shows then that the homotopy fiber of the nullification map
$BX \rightarrow P_{\Sigma B\Z/p} BX$ is a single Eilenberg-Mac
Lane space $K(P, 2)$, where $P$ is an abelian $p$-torsion group.
Moreover, he also shows that the corresponding fibration is
principal. In particular, it implies that $P_{\Sigma B\Z/p} BX$ is
a $p$-good space, and $(P_{\Sigma B\Z/p} BX)^\wedge_p$ is
$p$-complete by the fibre lemma \cite[II.5.1]{MR51:1825}.

{}From the equivalence $P_{B\Z/p} X \simeq \Omega P_{\Sigma B\Z/p}
BX$, \cite[Theorem~3.A.1]{Dror}, we obtain a loop fibration $K(P,
1) \rightarrow X \rightarrow P_{B\Z/p} X$. Since $X$ is a loop
space with finitely generated mod $p$ cohomology, we know from
\cite[Theorem~7.3]{MR2264802}, or directly from
\cite{MR2002g:55016}, that $P$ is a finite direct sum of copies of
cyclic groups and Pr\"ufer groups and that $H^*(P_{B\Z/p} X;
\F_p)$ is finite.

Let us consider the loop space $P_{B\Z/p} X$. Notice that $\pi_1
P_{\Sigma B\Z/p} BX \cong \pi_1 BX$, which must be a finite
$p$-group by Remark \ref{p-good}. By $p$-completing we obtain
hence a $p$-compact group $BY = (P_{\Sigma B\Z/p} BX)^\wedge_p$.
\end{proof}

The fibration we have obtained allows us to give a precise
description of the component of the constant in the pointed mapping
space $\map_*(B\Z/p, BX)$.

\begin{corollary}
\label{cor acyclicmap}
Let $(X, BX, e)$ be a $p$-Noetherian group. Then $\map_*(B\Z/p,
BX)_c$ is the classifying space of a finite elementary abelian
$p$-group.
\end{corollary}

\begin{proof}
Consider the fibration $K(P, 2)^\wedge_p \rightarrow BX
\rightarrow BY$ from Theorem~\ref{structure of BX}. Since $BY$ is
a $p$-compact group, $H^*(\Omega BY;\F_p)$ is finite and
$\map_*(B\Z/p, \Omega BY)\simeq *$ by \cite{Miller}. Therefore the
component $\map_*(B\Z/p, BY)_c$ is contractible and $\map_*(B\Z/p,
BX)_c \simeq \map_*(B\Z/p, K(P, 2)^\wedge_p)_c$. By
\cite[Theorem~1.5]{Miller}, $\map_*(B\Z/p, K(P, 2)^\wedge_p)\simeq
\map_*(B\Z/p, K(P, 2))$, which has trivial homotopy groups in
degrees~$\geq 2$. The component of the constant map is thus the
classifying space of a finite elementary abelian $p$-group
$V=\Hom(\Z/p, P)$.
\end{proof}

{}From Theorem~\ref{structure of BX} we deduce that many
$p$-Noetherian groups are $3$-connected covers of $p$-compact
groups.

\begin{corollary}
\label{4connectedBX}
Let $(X, BX, e)$ be a $p$-Noetherian group. Then $X$ is
$3$-connected if and only if $BX$ is the $4$-connected cover of the
classifying space of a $p$-compact group.
\end{corollary}

\begin{proof}
One implication is obvious. Let us hence assume that $X$ is
$3$-connected and consider the fibration $K(P, 2)^\wedge_p
\rightarrow BX \rightarrow BY$ from Theorem~\ref{structure of BX}.
We see that $BY$ is $2$-connected, hence
$3$-connected~\cite[Theorem~6.10]{MR23:A2201}. This shows that $P$
must be a divisible abelian $p$-group, or equivalently that $K(P,
2)^\wedge_p \simeq K(\oplus \Z^\wedge_p, 3)$.
\end{proof}

To the Lie group Example~\ref{Harada-Kono} we can add now new
examples of $p$-Noetherian groups, namely those given by
$3$-connected covers of exotic $p$-compact groups.

\begin{example}
\label{exotic-p-compact-group}
{\rm  Let $X$ be a $p$-compact group such that $BX$ is
$3$-connected and $\pi_4(BX)\cong \Z\pcom$. By looking at the
classification of $p$-compact groups, we observe that there are
only two sporadic examples, namely numbers $23$ and $30$ in the
Shephard-Todd list \cite{MR0059914}, and one infinite family,
number $2b$, corresponding to the dihedral groups $D_{2m}$. The
triple $(X\langle 3 \rangle, (BX)\langle 4 \rangle, e)$ is a
$p$-Noetherian group by Corollary~\ref{4connectedBX}.

The two sporadic examples are defined at primes $p \equiv 1,4$ mod
$5$, and they are non-modular since the only primes which divide the
order of their Weyl group are $2$, $3$ and $5$.  The family of
$p$-compact groups corresponding to the dihedral groups $D_{2m}$ is
defined for primes $p\equiv \pm 1$ mod $m$. Note that $p=2$ occurs
when $m=3$ and corresponds to the exceptional Lie group~$G_2$.}
\end{example}

\begin{remark}
\label{classification}
{\rm From Corollary~\ref{4connectedBX} we obtain a classification
of $3$-connected $p$-Noetherian groups. They are given by the
$3$-connected covers of simply connected $p$-compact groups, which
are known from the recent classification results,
\cite{MR2314074}, \cite{MR2338539}, \cite{classification2}, and
\cite{classificationp}. A general classification will be more
difficult to obtain, even in the $2$-connected case, as there are
$p$-Noetherian groups fibering over a product of $p$-compact
groups which do not split themselves as a product. Consider indeed
the homotopy fiber of the composite map
$$
f: BS^3 \times BS^3 \rightarrow  K(\Z, 4) \times K(\Z, 4)
\rightarrow K(\Z, 4),
$$
where the first map is the fourth Postnikov section and the second
is given by the sum. Let us complete this fiber at the prime $7$
for example and call it $BX$. Even though $X$ splits as a product
$(S^3)^\wedge_7 \times (S^3)^\wedge_7 \langle 3 \rangle$, the
classifying space $BX$ does not split.

Assume that $BX$ splits as a product $(BS^3)^\wedge_7 \times
(BS^3)^\wedge_7\langle 4 \rangle$. There exists then an essential
map $g: (BS^3)^\wedge_7 \rightarrow (BS^3)^\wedge_7 \times
(BS^3)^\wedge_7$ such that $f \circ g \simeq *$ and $p_1 \circ g$
is an equivalence, where $p_1$ denotes the projection on the first
factor. But $g$ induces on the fourth homology group a map of
degree $n\neq 0$ on the first copy of $(BS^3)^\wedge_7$ and of
degree $m$ on the second. The composite $f \circ g$ will thus have
degree $n+m$ on $H_4$. We claim that this cannot be zero. Both $m$
and $n$ must be squares in $\Z^\wedge_7$ as a self-map of
$(BS^3)^\wedge_7$ is induced by a self-map on the maximal
torus~$(BS^1)^\wedge_7$. But the sum of two $7$-adic squares is
nul if and only both are so. Therefore $BX$ cannot split as a
product. We refer the reader to Dwyer and Mislin's
article~\cite{MR928824} for a complete study of self-maps
of~$BS^3$.}
\end{remark}

\section{Splitting fibrations and mapping spaces}
\label{sec splitting}
Our next aim is to obtain conditions under which the total space of
the evaluation fibration $\map_*(B\Z/p, BX)_c \rightarrow \map(B\Z/p,
BX)_c \rightarrow BX$ splits as a product. The key observation is that
$\map_*(B\Z/p, BX)_c$ has a single non-trivial homotopy group, see
Corollary~\ref{cor acyclicmap}.

We thus consider a fibration $F \rightarrow E \rightarrow B$ of
connected spaces and assume that the homotopy fiber $F$ has finitely
many homotopy groups (it is a Postnikov piece, or in other words
there exists an integer $n$ such that the $n$-th Postnikov section
$F \rightarrow F[n]$ is a homotopy equivalence). Such a fibration is
classified by a map $B \rightarrow B\aut(F)$, where $\aut(F)$
denotes the monoid of self-equivalences of $F$. The original
fibration can be recovered by pulling-back the universal fibration
$F \rightarrow B\aut_*(F) \rightarrow B\aut(F)$, where $\aut_*(F)$
is the monoid of pointed self-equivalences. The existence of the
universal fibration was known to Dold, \cite{MR0198464}, but it was
Gottlieb who identified first the total space, \cite{MR0331384}.

\begin{proposition}
\label{prop Jesper}
Let $F \simeq F[n]$ be a connected Postnikov piece and $X$ be any
space. Then
\begin{itemize}
\item[(1)] $\map(X, F) \simeq (\map(X, F))[n]$ and $\pi_n(\map(X, F))\cong \pi_n(F)$,

\item[(2)] $\map_*(X, F) \simeq (\map_*(X, F))[n-1]$.
\end{itemize}
\end{proposition}

\begin{proof}
We proceed by induction on the number of non-trivial homotopy
groups of $F$. When $F$ is a $K(G, 1)$, $\map(X, F) \simeq
\map(X[1], F)$ and any component has the homotopy type of the
classifying space of a centralizer. Likewise $\map_*(X, K(G, 1))$
is homotopically discrete, the components being in bijection with
homomorphisms $\pi_1 X \rightarrow G$.

Suppose now that $n \geq 2$, write $A = \pi_n F$, and consider the
fibration $K(A, n) \rightarrow F \xrightarrow{p} F[n-1]$. Let us
fix a map $k: X \rightarrow F$. We analyze one component of the
mapping spaces at a time. Observe first that both claims are true
for the component of the constant map since we have a fibration
$\map(X, K(A, n)) \rightarrow \map(X,F) \rightarrow
\map(X,F[n-1])$ and can conclude by the classical result of Thom,
\cite{MR0089408}, and Federer, \cite{MR0079265}, that $\map(X,
K(A, n))$ splits as a product of Eilenberg-Mac Lane spaces $\prod
K(H^i(X; A), n-i)$. M\o ller has proved in \cite{MR910659} a
relative version of this lemma which will allow us to understand
the other components as well.

Let $\bar k$ denote the composite $X \rightarrow F \rightarrow
F[n-1]$. The $(n-1)$-st Postnikov section induces a fibration
$\map(X, F)_k \rightarrow \map(X, F[n-1])_{\bar k}$, the fiber of
which is  $F(X, \emptyset; F, F[n-1])_{\bar k}$, i.e. the space of
all lifts $f: X \rightarrow F$ such that $p \circ f = \bar k$.
This space of lifts is a product of Eilenberg-Mac Lane spaces by
\cite[Theorem~3.1]{MR910659} just like in the case of the
component of the trivial map. In particular the highest
non-trivial homotopy group of $\map(X, F)_k$ is the $n$-th,
isomorphic to~$A$.

In the case of pointed mapping spaces the space of lifts $F(X, *;
F, F[n-1])_{\bar k}$ appears in a similar argument. Its highest
non-trivial homotopy group is the $(n-1)$-st one.
\end{proof}

\begin{remark}
{\rm The proof of Proposition \ref{prop Jesper} also shows that if
$X$ is $n$-connected then $\map_*(X,F)$ is contractible.}
\end{remark}

When $F = K(G, 1)$ has a single homotopy group, it is well-known
that $\aut(F)$ is the semi-direct product $K(G, 1) \ltimes
\Aut(G)$ and $\aut_*(F) \simeq \Aut(G)$. All higher homotopy
groups of the topological monoid of self-equivalences are trivial.
The following corollary generalizes this observation.

\begin{corollary}
\label{cor Postnikov aut}
Let $n \geq 1$ and $F \simeq F[n]$ be a connected Postnikov piece.
Then
\begin{itemize}
\item[(1)] $B\aut(F) \simeq (B\aut(F))[n+1]$ and $\pi_{n+1}(B\aut(F))=\pi_n(F)$,

\item[(2)] $B\aut_*(F) \simeq (B\aut_*(F))[n]$.
\end{itemize}
\end{corollary}

\begin{proof}
The components of the monoid $\aut(F)$ are precisely the
components in $\map(F, F)$ of maps which are homotopy
equivalences. We conclude by Proposition~\ref{prop Jesper}.
\end{proof}

This statement is very much related to the work of Dwyer, Kan, and
Smith, \cite{MR974710}, where they provide classifying spaces for
towers of fibrations with prescribed fibers. In our case these
fibers would be $K(\pi_1 F, 1), \ldots, K(\pi_n F, n)$.

Let us now come back to our fibration $F \rightarrow E \rightarrow
B$. We ask when the total space $E$ splits as a product $B \times
F$.

\begin{theorem}
\label{thm splitting}
Let $F \rightarrow E \rightarrow B$ be a fibration of connected
spaces and assume that $F \simeq F[n]$.
\begin{itemize}
\item[(1)] If $B$ is $n$-connected, then $E \simeq B \times F$ if and only if
 the connecting morphism $\pi_{n+1}(B)\rightarrow \pi_n(F)$ is trivial.

\item[(2)] If the fibration has a section $s: B \rightarrow E$ and $B$ is
$(n-1)$-connected, then $E \simeq B \times F$ if and only if the
morphism $\pi_n(B)\rightarrow \pi_n(B\aut_*(F))$ is trivial.
\end{itemize}
\end{theorem}

\begin{proof}
The fibration is classified by a map $f: B \rightarrow B\aut(F)$
and $\pi_i B\aut(F) = 0$ for $i > n+1$ by Corollary~\ref{cor
Postnikov aut}. Part (1) when $B$ is $(n+1)$-connected is thus a
direct consequence of the first part of this corollary. Now assume
that $B$ is $n$-connected, consider the pullback diagram
\[
\diagram
F \dto \ar@{=}[r] & F\dto & \\
E' \dto \rto & E \dto \rto & K(\pi_{n+1}(B), n+1) \ar@{=}[d]\\
B\langle n+1 \rangle \rto^i & B \rto & K(\pi_{n+1}(B), n+1)
\enddiagram
\]
where the left vertical fibration splits since the base now is
$(n+1)$-connected. That means that $f$ restricted to $B\langle n+1
\rangle$ is null-homotopic. Since $\map_*(B\langle n+1 \rangle,
B\aut(F))_c$ is contractible (Proposition \ref{prop Jesper}),
applying the Zabrodsky Lemma, \cite[Proposition~3.4]{MR97i:55028},
we deduce that $f$ factors through $K(\pi_{n+1}(B), n+1)$.
Therefore, to show that $f$ is null-homotopic, we only need to
prove that the induced map on homotopy groups $\pi_{n+1}(B)
\rightarrow \pi_{n+1} (B\aut(F))\cong \pi_n(F)$ is trivial. The
naturality of the long exact sequence on homotopy groups shows
that this morphism is the connecting morphism for the fibration $F
\rightarrow E \rightarrow B$.

To prove (2) we use the fact that $f$ factors actually through a
map $B \rightarrow B\aut_*(F)$ if the fibration has a section. We
first assume that $B$ is $n$-connected. In this case, from the
connectivity assumption on $B$ and part (2) in Corollary~\ref{cor
Postnikov aut} we see that $f$ is null-homotopic. If $B$ is
$(n-1)$-connected, consider the fibration $B\langle n \rangle
\rightarrow B \rightarrow K(\pi_n(B), n)$. The same argument as
above shows that $f$ factors through $K(\pi_n(B), n)$. Therefore,
$f$ is null-homotopic if and only if the induced map on homotopy
groups $\pi_n(B) \rightarrow \pi_n (B\aut_*(F))$ is trivial.
\end{proof}

\begin{corollary}
\label{cor splitting n=1}
Let $K(G,1) \rightarrow E \rightarrow B$ be a fibration of
connected spaces with a section where $G$ is a discrete group.
Then $E\simeq K(G,1)\times B$ if and only if the induced action
$\pi_1(B)\rightarrow \Aut(G)$ is trivial.
\end{corollary}

\begin{proof}
Note that $\aut_*(K(G,1)) \simeq \Aut(G)$.
\end{proof}

\begin{corollary}
\label{cor splitting}
Let $n \geq 0$, $A$ be a connected space, and $B$ be an
$n$-connected space such that $\Omega^{n+1} B$ is $A$-local. Then
the homotopy groups of $\map_*(A,B)_c$ are concentrated in degrees
from $1$ to $n$ and $\map(A, B)_c \simeq \map_*(A, B)_c \times B$.
\end{corollary}

\begin{proof}
Let us consider the evaluation fibration $\map_*(A, B)_c
\rightarrow \map(A, B)_c \rightarrow B$. It has always a section,
given by the constant maps. Since $\Omega^{n+1} B$ is $A$-local,
we see that $\Omega^{n+1} \map_*(A, B)_c$, being weakly equivalent
to $\map_*(A, \Omega^{n+1} B)$, is contractible. Therefore the
homotopy groups of $\map_*(A, B)_c$ in degree $>n$ are all
trivial. Part (2) of Theorem~\ref{thm splitting} applies.
\end{proof}

\begin{remark}
\label{splitting}
{\rm The sharpness of the connectivity assumption in
Theorem~\ref{thm splitting} is illustrated by the following
non-trivial fibrations. For part (1), consider the fibration
$K(\Z, n) \rightarrow S^{n+1}\langle n+1 \rangle \rightarrow
S^{n+1}$. The base space is only $n$-connected, and the fibration
is not trivial, as it is classified by a non-trivial map $S^{n+1}
\rightarrow K(\Z, n+1)$. In this example the connecting morphism
$\pi_{n+1}(S^{n+1})\rightarrow \pi_n(K(\Z,n))$ is an isomorphism.

For part (2), let $p$ be any prime, $n \geq 2$ and $F$ be the
product $K(\Z/p,1)\times K(\Z/p, n)$. Since $F$ is an $H$-space,
the identity component $\map_*(F,F)_{id}$ is weakly equivalent to
the component of the constant map, i.e. is a product of
Eilenberg-MacLane spaces, one of them being $K(\Z/p, n-1)$.
Therefore $\pi_n(B\aut_*(F)) \cong \Z/p$ and there exists a map
$S^n \rightarrow B\aut_*(F)$ classifying a split fibration over
$S^n$ with fiber $F\simeq F[n]$, which is not trivial. In this
example $\pi_n(S^n)\cong \Z \rightarrow \pi_n(B\aut_*(F)) \cong
\Z/p$ is the projection.}
\end{remark}

We apply now the results of this section to analyze certain
mapping spaces. This allows us in particular to understand the
space $\map(B\Z/p,BX)_c$ for any $p$-Noetherian group~$X$.

\begin{proposition}
\label{prop general nonconnectedsplitting}
Let $Z$ be a space such that $\Omega^2 Z$ is $B\Z/p$-local. Then
the component of the mapping space $\map(B\Z/p,Z)_c$ splits as a
product $Z \times \map_*(B\Z/p,Z)_c$ and $\map_*(B\Z/p,Z)_c$ is
the classifying space of an elementary abelian $p$-group (not
necessarily finite).
\end{proposition}

\begin{proof}
By the work of Bousfield, \cite[Theorem~7.2]{B2}, the homotopy
fiber of the nullification map $Z\rightarrow P_{\Sigma B\Z/p}Z$ is
a single Eilenberg-Mac Lane space $K(P, 2)$, where $P$ is an
abelian $p$-torsion group. He also shows that the fibration
$K(P,2) \rightarrow  Z \rightarrow   P_{\Sigma B\Z/p}Z$ is
principal. By adjunction, the component $\map_*(B\Z/p, P_{\Sigma
B\Z/p}Z)_c$ is contractible. Therefore $\map_*(B\Z/p, Z)_c \simeq
\map_*(B\Z/p, K(P, 2))_c$, which is the classifying space of the
elementary abelian $p$-group $W=\Hom(\Z/p, P)$.

Now, by Corollary \ref{cor splitting n=1}, we only need to check
that the action of $\pi=\pi_1(Z)$ on $W$ is trivial. By taking
mapping spaces at the component of the constant map and
evaluation, we obtain the following diagram of fibrations (since
$\map_*(B\Z/p, P_{\Sigma B\Z/p}Z)_c$ is contractible):
\[
\diagram
 BW \dto \ar@{=}[r] &  BW \dto & \cr
\map(B\Z/p,K(P,2))_c \rto \dto & \map(B\Z/p,Z)_c \dto \rto &
\map(B\Z/p,P_{\Sigma B\Z/p}Z)_c \dto^{\simeq} \cr K(P,2) \rto & Z \rto & P_{\Sigma B\Z/p}Z
\enddiagram
\]
The bottom and middle horizontal fibrations are principal,
therefore the action of the fundamental group of the base space,
$\pi_1 P_{\Sigma B\Z/p}Z \cong \pi$, is trivial on all homotopy
groups of the fiber, in particular on the fundamental group of the
fiber. This action can be seen as conjugation in the fundamental
group of the total space $\map(B\Z/p,Z)_c$, but now it does not
matter whether we look at the vertical fibration or the horizontal
one (in both cases the induced morphism is surjective on the
fundamental group).
\end{proof}

\begin{corollary}
\label{prop nonconnectedsplitting}
Let $X$ be a $p$-Noetherian group. Then the mapping space
$\map(B\Z/p,BX)_c$ splits as a product $BX \times
\map_*(B\Z/p,BX)_c$ where $\map_*(B\Z/p,BX)_c$ is the classifying
space of a finite elementary abelian $p$-group. In particular, the
mapping space $\map(B\Z/p,BX)_c$  is $p$-good, $p$-complete and
$H^*(\map(B\Z/p,BX)_c;\F_p)$ is of finite type.
\end{corollary}

\begin{proof}
The finiteness of the elementary abelian $p$-group follows from
Corollary~\ref{cor acyclicmap}.
\end{proof}

In particular, we see that $\map(B\Z/p,BX)_c$ is again the
classifying space of a $p$-Noetherian group.

\section{Indecomposable elements and the Krull filtration}
\label{sec QandU}
As mentioned in the introduction, a good way to understand the
cohomology of a space as an algebra over the Steenrod algebra is
to look at the module of indecomposable elements $QH^*(X; \F_p) =
\tilde H^*(X; \F_p) / \tilde H^*(X; \F_p) \cdot \tilde H^*(X;
\F_p)$. An important observation here is that this definition
depends on the choice of a base point, or more exactly on the
choice of a component $X_0$ if $X$ is not connected. Since
$H^0(X;\F_p)$ is a $p$-Boolean algebra, it follows that $QH^*(X;
\F_p)$ is isomorphic to $QH^*(X_0; \F_p)$.

There is a (Krull) filtration of the category
$\U$ of unstable modules, $\U_0 \subset \U_1 \subset \dots$ such
that $\U_0$ consists in the locally finite unstable module. Schwartz
established in \cite[Theorem~6.2.4]{MR95d:55017} a criterion to
check whether (and where) an unstable module lives in the Krull
filtration, namely $M \in \U_n$ if and only if $\bar T^{n+1} M = 0$,
where $\bar T$ is Lannes' reduced $T$-functor.

Therefore our objective in this section is to prove that $\bar T^2
QH^*(BX; \F_p) = 0$ for a $p$-Noetherian group. To do so, we need
first to find a geometrical interpretation of the reduced
$T$-functor.

Recall that ``under some mild assumptions", $TH^*(Z; \F_p) \cong
H^*(\map (B\Z/p, Z); \F_p)$. Lannes' standard mild assumptions on
$Z$  are that $T H^*(Z; \F_p)$ is of finite type (or
$H^*(\map(B\Z/p, Z); \F_p)$ is of finite type), and that
$\map(B\Z/p, Z)$ is $p$-good,
\cite[Proposition~3.4.4]{MR93j:55019}. We will not need to
understand globally the mapping space, but restrict our attention
to the component $\map (B\Z/p, Z)_c$ of the constant map, the
natural choice of base point in the full mapping space. We thus
only consider the component $T_c(H^*(Z; \F_p))$ of Lannes'
$T$-functor.

\begin{theorem}
\label{thm TQ}
Let $Z$ be a $p$-complete space such that $H^*(Z;\F_p)$ and
$H^*(\map(B\Z/p,Z)_c)$ are of finite type. If $\Omega^2 Z$ is
$B\Z/p$-local, then $$\bar T QH^*(Z; \F_p) \cong QH^*
(\map_*(B\Z/p, Z)_c; \F_p).$$ In particular the unstable module
$QH^*(Z; \F_p)$ lies in $\mathcal U_{1}$.
\end{theorem}

\begin{proof}
In Proposition~\ref{prop general nonconnectedsplitting} we
obtained a splitting $\map(B\Z/p,Z)_c \simeq
\map_*(B\Z/p,Z)_c\times Z$ and an equivalence $\map_*(B\Z/p,Z)_c
\simeq BW$ where $W$ is an elementary abelian $p$-group. With the
hypothesis of the theorem this splitting shows that $H^*(BW;\F_p)$
is of finite type and therefore $W$ is finite. Therefore
$\map(B\Z/p, Z)_c$ is $p$-good. Since moreover
$H^*\bigl(\map(B\Z/p, Z)_c; \F_p\bigr)$ is of finite type by
assumption, we can apply Lannes' result
\cite[Proposition~3.4.4]{MR93j:55019} and deduce that the
$T$-functor computes what it should: $T_c H^*(Z; \F_p) \cong H^*
(\map(B\Z/p, Z)_c; \F_p)$.

Notice also that $QH^* (\map(B\Z/p, Z); \F_p) \cong QH^*
(\map(B\Z/p, Z)_c; \F_p)$. Since Lannes' $T$-functor commutes with
taking the module of indecomposable elements, $T Q H^*(Z; \F_p)
\cong QTH^*(Z; \F_p)$. But in degree zero $TH^*(Z; \F_p)$ is a
Boolean algebra, \cite[Section~3.8]{MR95d:55017}, so that
$QTH^*(Z; \F_p) \cong Q\bigl(T_c H^*(Z; \F_p)\bigr)$, which is
isomorphic to $QH^* (\map(B\Z/p, Z)_c; \F_p)$. The splitting
yields next an isomorphism
$$
T Q H^*(Z; \F_p) \cong QH^* (\map_*(B\Z/p, Z)_c; \F_p) \oplus
QH^*(Z; \F_p)
$$
so that we have finally identified $\bar T QH^*(Z; \F_p) \cong
QH^* (\map_*(B\Z/p, Z)_c; \F_p)$. This proves the first part of
the theorem. For the second claim, use the fact that
$\map_*(B\Z/p, Z)_c \simeq BW$, the classifying space of a finite
elementary abelian group. The cohomology of $W$ is finitely
generated as an algebra, so $Q H^*(BW; \F_p)$ is finite and lies
in $\U_0$. Therefore $\bar TQ H^*(BW; \F_p) = 0$, or equivalently
$\bar T^2 QH^*(Z; \F_p) = 0$, and so $QH^*(Z; \F_p)$ lies in
$\mathcal U_{1}$.
\end{proof}

Let us now turn to an even finer analysis of the module of
indecomposable elements. Let us denote by $Q_1$ the unstable
module $QH^*(B\Z/p; \F_p)$ of the cohomology of a cyclic group of
order~$p$. At the prime $p=2$, the unstable module $Q_1$ is
isomorphic to $\Sigma \F_2 = \Sigma F(0)$. At an odd prime $Q_1$
is an unstable module with one generator $t$ in degree~1 and its
Bockstein $\beta t$ in degree~2.

\begin{proposition}
\label{prop schwartz}
Let $Z$ be a $p$-complete space such that $H^*(Z;\F_p)$ and
$H^*(\map(B\Z/p,Z)_c)$ are of finite type. Assume that  $\Omega^2
Z$ is $B\Z/p$-local. Define $Q_1 = QH^*(B\Z/p; \F_p)$. Then there
exists a morphism $QH^*(Z; \F_p) \rightarrow F(1)
\otimes(Q_1^{\oplus k})$ with finite cokernel and locally finite
kernel.
\end{proposition}

\begin{proof}
Schwartz characterizes in~\cite[Proposition~2.3]{MR2002k:55043} the
unstable modules $M$ in $\U_1$ as those sitting in an exact sequence
$0 \rightarrow K \rightarrow M \rightarrow F(1) \otimes L
\rightarrow N \rightarrow 0$, where $K, L$, and $N$ are locally
finite (i.e. in $\U_0$). In particular $\bar T M \cong L$ since
$\bar T F(1) = F(0)$ and $T$ commutes with tensor products,
\cite[Theorem~3.5.1]{MR95d:55017}. In our case we know from the
previous theorem that $\bar T QH^*(Z; \F_p) \cong QH^*(BW; \F_p)$
where $W$ is an abelian elementary group, say of rank~$k$. Thus $L =
Q_1^{\oplus k}$. The quotient $N$ of $F(1) \otimes(Q_1^{\oplus k})$
will be finitely generated. As it is locally finite it must be
finite.
\end{proof}

We finally come back to $p$-Noetherian groups and prove that the
module of indecomposable elements $QH^*(BX; \F_p)$ is as small as
expected.

\begin{theorem}
\label{thm TQforpNoetherian}
Let $X$ be a $p$-Noetherian group. Then $$\bar T QH^*(BX; \F_p) \cong
QH^* (\map_*(B\Z/p, BX)_c; \F_p).$$ In particular the unstable module
$QH^*(BX; \F_p)$ lies in $\mathcal U_{1}$.
\end{theorem}

\begin{proof}
The assumptions in Theorem~\ref{thm TQ} are satisfied by Corollary
\ref{prop nonconnectedsplitting}.
\end{proof}

\section{Fibrations over spaces with finite cohomology}
\label{sec Serre}
In our study of $H^*(BX; \F_p)$, we have already managed to prove
that $QH^*(BX; \F_p)$ lives in $\U_1$, that is only one stage
higher than where $QH^*(X; \F_p)$ lives. What is left to prove is
that $H^*(BX; \F_p)$ is finitely generated as an algebra over the
Steenrod algebra. Therefore we analyze the fibration $K(P,
2)^\wedge_p \rightarrow BX \rightarrow BY$ of
Theorem~\ref{structure of BX}.

Let $F\rightarrow E \rightarrow B$ be a fibration where both $H^*(B;
\F_p)$ and $H^*(F; \F_p)$ are finitely generated $\A_p$-algebras. In
this situation, we ask whether the same finiteness condition holds
for $H^*(E; \F_p)$. When the fibration is one of $H$-spaces and
$H$-maps we proved in \cite{CCS5} that this is true. But in general
some restrictions have to be imposed, even when the fiber is a
single Eilenberg-Mac Lane space as shown by the following example.

\begin{example}
\label{counter--example}
{\rm Consider the folding map $K(\Z, 3) \vee K(\Z, 3) \rightarrow
K(\Z, 3)$. An easy application of Puppe's theorem \cite{MR51:1808},
shows that the homotopy fiber is $\Sigma \Omega K(\Z, 3) \simeq
\Sigma K(\Z, 2)$. Therefore there exists a fibration
\[
K(\Z, 2) \rightarrow \Sigma K(\Z, 2) \rightarrow K(\Z, 3) \vee K(\Z,
3).
\]
The mod $2$ cohomology of the fiber is finitely generated as an
algebra, the cohomology of the base space is generated over
$\mathcal A_2$ by the two fundamental classes in degree $3$. However
the cohomology of $\Sigma K(\Z, 2)$ is not finitely generated over
$\mathcal A_2$ (as it is a suspension it would be finitely generated
as an unstable module, and therefore would belong to some stage of
the Krull filtration; by Schwartz's solution \cite{MR99j:55019} to
Kuhn's non-realizability conjecture this would imply that the
cohomology were locally finite, which it is not).}
\end{example}

This example indicates that we must impose stronger conditions on
the base space of the fibration to make sure that the cohomology
of the total space is finitely generated as an algebra
over~$\A_p$. In this section we study fibrations $F \rightarrow E
\rightarrow B$ where $\pi_1 B$ acts trivially on the cohomology of
the fiber. We will assume that $H^*(B; \F_p)$ is \emph{finite} and
the fiber $F$ is a finite product of Eilenberg Mac-Lane spaces
$\prod_{i=1}^q K(A_i, n_i)$ where $A_i$ is a finitely generated
abelian group for all~$i$. Both assumptions will play an essential
role in the analysis of the cohomology of the total space. The
finiteness of the base forces the Serre spectral sequence to
collapse at some finite stage and the hypothesis on $A_i$ implies
that the cohomology of $K(A_i, n_i)$ is generated, as an algebra
over the Steenrod algebra $\mathcal A_p$, by a finite number of
fundamental classes $\iota_1, \dots, \iota_m$ of degree $n$, and
possibly certain higher Bockstein on these classes. It is a free
algebra by work of Serre at the prime $2$, \cite{MR0060234}, and
Cartan at odd primes, \cite{MR0087934}.

\begin{lemma}
\label{SSsplitting}
There exists a splitting $H^*\bigl(\prod K(A_i, n_i); \F_p\bigr)
\cong F^* \otimes G^*$ of algebras where $F^*$ is finitely generated
as an algebra, and $G^*$ consists of permanent cycles in the Serre
spectral sequence. Moreover $G^*$ is finitely generated as an
algebra over $\mathcal A_p$.
\end{lemma}

\begin{proof}
By Kudo's transgression theorem, all classes obtained by applying
Steenrod operations to transgressive operations are transgressive.
Let us choose therefore an integer $r$ larger than the dimension of
the cohomology of the base. If $\{x_1,\ldots, x_k\}$ is a set of
generators of $H^*(F; \F_p)$ as an $\A_p$-algebra, the elements $1
\otimes \mathcal P^I x_k$ are permanent cycles for any sequence $I$
of degree larger than $r-n-1$ and any $k$.

We will say that such generators $\mathcal P^I x_k$ have
\emph{large} degree and the others, of which there is only a
finite number, have small degree. We define now $F^*$ to be the
subalgebra generated by the generators of small degree and $G^*$
by all other large degree generators. Then $H^*(F; \F_p)\cong
F^*\otimes G^*$.

The last claim is proven by looking at the inclusion of algebras
$G^* \subset H^*(F; \F_p)$. At the level of modules of
indecomposable elements it induces an inclusion $QG^* \subset
QH^*(F; \F_p)$, because of the freeness of $H^*(F; \F_p)$ and our
choices of generators. Since the category $\mathcal U$ of unstable
modules is locally Noetherian, \cite[Theorem~1.8.1]{MR95d:55017},
the unstable module $QG^*$ is finitely generated. Therefore, $G^*$
is a finitely generated $\A_p$-algebra.
\end{proof}

The proof of the next proposition follows the line of the
Dwyer-Wilkerson result \cite[Proposition~12.4]{MR1274096}, see
also Evens, \cite{MR0137742}.

\begin{proposition}
\label{fgmodule}
The cohomology of the total space $H^* (E; \F_p)$ is finitely
generated as a module over $H^*(B; \F_p)[z_1, \dots, z_k] \otimes
G^*$.
\end{proposition}

\begin{proof}
The free algebra $F^*$ is finitely generated and we consider first
all polynomial generators $a_1, \dots, a_k$. We define $z_i =
(a_i)^{p^{n_i}}$ where $n_i$ is the smallest integer such that this
power of $a_i$ is a permanent cycle (it exists since these powers
are transgressive, compare with the proof of
Lemma~\ref{SSsplitting}). Then $F^*$ is a finitely generated module
over $\F_p[z_1, \dots, z_k]$. Chose now a finite set of generators
$g_1, \dots , g_r$ of $G^*$ as an algebra over the Steenrod algebra.

The elements $z_i$ and the elements $g_1, \dots , g_r$ are permanent
cycles in the vertical axis of the Serre spectral sequence, one can
thus choose elements $z_i'$ and $g_j'$ in $H^* (E; \F_p)$ whose
images in $H^* (K(A, n); \F_p)$ are the $z_i$'s and the $g_j'$'s.
Better said, since both $\F_p[z_i]$ and $G^*$ are free algebras, we
choose an algebra map $s\colon \F_p[z_i] \otimes G^* \hookrightarrow
H^*(E; \F_p)$. The elements in $H^* (B; \F_p)$ act on $H^* (E;
\F_p)$ via $p^*\colon H^* (X; \F_p) \rightarrow H^* (E; \F_p)$. This
explains the module structure.

We see that $E_\infty = E_r$ is finitely generated as a module
over $H^*(B; \F_p)[z_1, \dots, z_k] \otimes G^*$. Therefore so is
$H^*(E; \F_p)$ by \cite[Corollary~VII.3.3]{MR0389876}.
\end{proof}

The difficulty to infer information about the $\A_p$-algebra
structure from the module structure is that the algebra map $s$ is
not a map of $\mathcal A_p$-algebras. To circumvent this problem we
will appeal to the algebraic result proved in the
appendix~\ref{appendix}.

\begin{theorem}
\label{thm SSS}
Consider a fibration $\prod_{i=1}^q K(A_i, n_i) \rightarrow E
\rightarrow B$ where $H^*(B; \F_p)$ is finite and $A_i$ is a
finitely generated abelian group for all~$i$. The cohomology $H^*(E;
\F_p)$ is then finitely generated as an algebra over~$\A_p$.
\end{theorem}

\begin{proof}
In the notation of the appendix, $H^*(B; \F_p)[z_1, \dots, z_k]$
is the connected and commutative finitely generated algebra $C^*$,
and $B^* = H^*(E; \F_p)$. By Proposition~\ref{fgmodule}, $H^*(E;
\F_p)$ is a finitely generated $C^* \otimes G^*$-module. The
action of $C^* \otimes G^*$ on $B^*$ has been defined in the
previous proof via an algebra map (constructed from a section $s:
G^* \rightarrow B^*$), thus $B^*$ is a $C^* \otimes G^*$-algebra.
Define now $\pi: H^*(\prod K(A_i,n_i); \F_p) \cong F^* \otimes G^*
\rightarrow G^*$ to be the projection and $p: B^* \rightarrow G^*$
to be the composite $\pi \circ i^*$. This is a morphism of
$G^*$-modules so that Proposition~\ref{algebraprop} applies. Hence
$H^*(E; \F_p)$ is finitely generated as an algebra over~$\A_p$.
\end{proof}

\begin{remark}
\label{rem generalize}
{\rm The nature of Theorem~\ref{thm SSS} is purely cohomological.
Therefore the same statement remains true if we relax the assumption
on the fiber in the following way: The fiber $F$ should be
homotopic, up to $p$-completion, to $\prod_{i=1}^q K(A_i, n_i)$
where each $A_i$ is a finitely generated abelian group. This will
allow us to include summands of the form $\Z_{p^\infty}$ or
$\Z^\wedge_p$. In fact the same proof goes through with the
assumption that $H^*(F; \F_p)$ is a free algebra, finitely generated
as an algebra over~$\A_p$.}
\end{remark}

As a byproduct we obtain the following result on highly connected
covers of finite complexes. For a mod $p$ finite $H$-space $B$ we
proved in \cite{CCS5} that the mod $p$ cohomology of $B\langle n
\rangle$ is finitely generated for \emph{any} integer~$n$.

\begin{corollary}
\label{cor finitecover}
Let $n \geq 2$ and $B$ be an $(n-1)$-connected space with finite mod
$p$ cohomology. Then $H^*(B\langle n\rangle; \F_p)$ is finitely
generated as an algebra over~$\A_p$. Moreover the unstable module of
indecomposable elements $QH^*(B\langle n\rangle; \F_p)$ lies in
$\mathcal U_{n-2}$.
\end{corollary}

\begin{proof}
The cohomology $H^*(B\langle n\rangle; \F_p)$ is finitely
generated as an algebra over $\A_p$ by a direct application of
Theorem~\ref{thm SSS}. Now from Corollary~\ref{cor splitting} we
infer that $\map(B\Z/p, B\langle n\rangle)_c$ splits as a product
$B\langle n\rangle \times \map_*(B\Z/p, B\langle n\rangle)_c$.
Since $B$ itself is $B\Z/p$-local by Miller's Theorem,
\cite{Miller}, this pointed mapping space is equivalent to
$\map_*(B\Z/p, K(\pi_n X, n-1)_c$. This is a product of
Eilenberg-Mac Lane spaces, the highest of them being $K(A, n-2)$
where $A = \Hom(\Z/p, \pi_n B)$. The module of indecomposable
elements of its cohomology lies in~$\mathcal U_{n-3}$. Thus
$QH^*(B\langle n\rangle; \F_p)$ lies in $\mathcal U_{n-2}$, the
proof is analogous to that of Theorem~\ref{thm TQ}.
\end{proof}

\section{The cohomology of $p$-Noetherian groups}
\label{sec pNoetherian_H^*}
We are about to conclude our study of the cohomology of
classifying spaces of $p$-Noetherian groups. We have seen that any
$p$-Noetherian group is the total space of a fibration over a
$p$-compact group with fiber an Eilenberg-Mac Lane space. Recall
from \cite{MR1274096}'s main theorem that the mod $p$-cohomology
of the classifying space of a $p$-compact group is finitely
generated as an algebra. Our objective is to prove the following
theorem, which together with Theorem~\ref{thm TQforpNoetherian}
gives a very accurate description for the cohomology of
classifying spaces of $p$-Noetherian groups.

\begin{theorem}
\label{theorem BX and Ap}
Let $(X, BX, e)$ be a $p$-Noetherian group. Then $H^*(BX; \F_p)$ is
finitely generated as an algebra over the Steenrod algebra.
\end{theorem}

The strategy is to use the fibration $K(P, 2)^\wedge_p \rightarrow
BX \rightarrow BY$ of Theorem~\ref{structure of BX}. As we do not
know whether the Serre spectral sequence collapses at some finite
stage, we reduce the problem in several steps to the study of a
spectral sequence over a finite base (in order to apply our
results from the previous section).

A $p$-compact toral group $P$ is a $p$-compact group which is an
extension of a $p$-compact torus $((S^1)^n)\pcom$ by a finite
$p$-group. Dwyer and Wilkerson in \cite{MR1274096} show that any
$p$-compact group $Y$ admits a maximal $p$-compact toral subgroup
$N\leq Y$ such that the homotopy fiber $Y/N$ of the map $Bi\colon
BN\rightarrow BY$ has finite mod $p$ cohomology and Euler
characteristic prime to $p$ (see \cite[Proof of 2.3]{MR1274096}).
A transfer argument (see \cite[Theorem~9.13]{MR1274096}) then
shows that $Bi$ induces a monomorphism in mod $p$ cohomology.
Consider now  the pullback diagram
\[
\diagram
\widetilde{BN} \dto \rto^{B\tilde{i}} & BX \dto \\
BN \rto^{Bi} & BY.
\enddiagram
\]

First note that, by Proposition \ref{closed-under-fibrations},
$\widetilde{BN}$ is the classifying space of a $p$-Noetherian
group because we have a fibration $K(P, 2)^\wedge_p \rightarrow
\widetilde{BN} \rightarrow BN$. We will first show that
$H^*(\widetilde{BN};\F_p)$ is a finitely generated $\A_p$-algebra
by using this fibration, and then we will show that the cohomology
$H^*(BX;\F_p)$ is a finitely generated $\A_p$-algebra by using the
fibration $Y/N \rightarrow \widetilde{BN} \rightarrow BX$. We
start with a technical result which will be used in both steps of
the proof.

\begin{proposition}
\label{fibration-local-fibre}
Let $F\rightarrow E\stackrel{q}{\rightarrow} B$ be a fibration
such that $\pi_1(B)$ acts nilpotently on $H_*(F;\F_p)$. Assume
that $q$ induces an isomorphism $\bar TQH^*B\cong \bar TQH^*E$,
$H^*(F;\F_p)$ is locally finite, and $H^*(E;\F_p)$ is a finitely
generated $\A_p$-algebra. Then, if $\ker q^*$ is a finitely
generated ideal then $H^*(B;\F_p)$ is a finitely generated
$\A_p$-algebra.
\end{proposition}

\begin{proof}
Let $A$ be the algebra which is the quotient of $H^*(E; \F_p)$
by the ideal generated by the image of $q^*$, that is,
$\F_p\otimes_{H^*(B; \F_p)}H^*(E; \F_p)$. There is a coexact
sequence
$$
H^*(B; \F_p)//\ker q^* \stackrel{q^*}{\rightarrow} H^*(E;
\F_p)\rightarrow A.
$$
Since $H^*(E; \F_p)$ is a finitely generated $\A_p$-algebra, the
same is true for $A$. Therefore $QA$ is a finitely generated
$\A_p$-module, which is the the cokernel of $Qq^*$ by
right-exactness of the functor $Q$. Moreover, since by assumption
$\bar TQH^*B\cong \bar TQH^*E$, it follows from exactness of the
reduced $T$ functor that $\bar T QA=0$. This means that $QA$
belongs to $\U_0$, i.e. it is locally finite, hence finite.
Equivalently $A$ is a finitely generated algebra.

We want to show that $A$ is in fact a finite algebra. The fiber
inclusion $F \rightarrow E$ of the fibration $q$ induces a
morphism $\iota^*\colon A\rightarrow H^*(F; \F_p)$. A careful
study of the Eilenberg-Moore spectral sequence,
\cite[Theorem~0.5]{MR827370} (for the prime $2$) or
\cite[Theorem~8.7.8]{MR95d:55017}, shows that this morphism is an
$F$-monomorphism. Take now any element $a\in A$. Because
$H^*(F;\F_p)$ is locally finite, there exists $M>0$ such that
$a^{p^M}\in \ker \iota^*$, which is nilpotent. There exists thus
$N>0$ such that $a^{p^{N+M}}=0$. Since all elements of $A$ are
nilpotent and it is a finitely generated algebra, $A$ must be
finite.

For the last statement, note that, as an $H^*(B; \F_p)$-module,
the cohomology $H^*(E; \F_p)$ is isomorphic to $\bigl( H^*(B;
\F_p)//\ker q^*\bigr )\{b_1,\ldots,b_k\}$, where the $b_i$'s are
the generators of $A$ as a (graded) $\F_p$-vector space. This
description shows that the morphism $Q(H^*(B; \F_p)//\ker q^*)
\rightarrow QH^*(E; \F_p)$ is an isomorphism in high degrees. That
is, there exists $K>0$ such that
$$
(Q(H^*(B; \F_p)//\ker q^*))^{>K}\cong (QH^*(E; \F_p))^{>K}\, ,
$$
which is an unstable submodule of $QH^*(E; \F_p)$. Since $H^*(E;
\F_p)$ is a finitely generated $\A_p$-algebra, $QH^*(E; \F_p)$ is
a finitely generated $\A_p$-module. But then, as the category of
unstable modules over the Steenrod algebra is locally Noetherian,
\cite[Theorem~1.8.1]{MR95d:55017}, the same holds for $(Q(H^*(B;
\F_p)//\ker q^*))^{>K}$ . In particular, $Q(H^*(B; \F_p)//\ker
q^*)$ is a finitely generated $\A_p$-module since it is of finite
type, that is $H^*(B; \F_p)//\ker q^*$ is a finitely generated
$\A_p$-algebra.

Finally, $H^*(B; \F_p)$ is generated by $H^*(B; \F_p)//\ker q^*$
and $\ker q^*$, it is  therefore also finitely generated as an
algebra over the Steenrod algebra since $\ker q^*$ is a finitely
generated ideal .
\end{proof}

We apply next this result to the fibration $Y/N \rightarrow
\widetilde{BN} \rightarrow BX$.

\begin{corollary}
$H^*(BX; \F_p)$ is a finitely generated algebra if
$H^*(\widetilde{BN}; \F_p)$ is so.
\end{corollary}

\begin{proof}
Consider the fibration $Y/N\rightarrow \widetilde{BN}
\stackrel{B\tilde i}{\rightarrow} BX$ and the diagram of
horizontal fibrations
\[
\diagram \map(B\Z/p, Y/N)_c \rto \dto^{ev}_{\simeq} &
\map(B\Z/p,\widetilde{BN})_c \dto^{ev_{r}} \rto & \map(B\Z/p,
BX)_c \dto^{ev_n} \cr Y/N \rto & \widetilde{BN} \rto & BX
\enddiagram
\]
where the left vertical arrow is an equivalence since the fiber
$Y/N$ is mod $p$ finite. Therefore $B\tilde i$ induces an
equivalence $\map_*(B\Z/p, \widetilde{BN})_c \rightarrow
\map_*(B\Z/p, BX)_c$. Both $BX$ and $\widetilde{BN}$ are
$p$-Noetherian groups and Theorem~\ref{thm TQforpNoetherian}
implies then that $\bar T QH^*BX \rightarrow \bar T
QH^*\widetilde{BN}$ is an isomorphism.

We conclude now by Proposition~\ref{fibration-local-fibre},
because $\pi_1(BX)$ is a finite $p$-group and $q^*$ is injective
by a transfer argument \cite[Theorem~9.13]{MR1274096} (the Euler
characteristic $\chi (Y/N)$ is prime to $p$).
\end{proof}

The key fact in the next argument is that any $p$-compact toral
group satisfies the Peter-Weyl theorem. That is, it admits a
homotopy monomorphism into $U(n)\pcom$ for some $n$. This is shown
for example in \cite[Proposition 2.2]{MR1600146}. Let us choose
then such a map $\rho: BN \rightarrow BU(n)\pcom$ for the maximal
$p$-compact toral group $BN$. The mod $p$ cohomology of the fibre
$F$ is hence $\F_p$-finite.

\begin{proposition}
\label{E0-fg}
Let $E_0$ be the homotopy pull-back of the diagram $F \rightarrow
BN \leftarrow \widetilde{BN}$. Then $H^*(E_0; \F_p)$ is finitely
generated as an algebra over the Steenrod algebra.
\end{proposition}

\begin{proof}
The space $E_0$ fits by construction into a fibration $K(P,2)\pcom
\rightarrow E_0 \rightarrow F$. Since the mod $p$ cohomology of
$F$ is finite, Theorem~\ref{thm SSS} shows that $H^*(E_0; \F_p)$
is a finitely generated $\A_p$-algebra.
\end{proof}

Let us now concentrate on the map $E_0 \rightarrow
\widetilde{BN}$. We will filter it by using the top left block
diagonal inclusions $U(r-1)\subset U(r)$ for $1 \leq r \leq n$. At
the level of classifying spaces these inclusions induce oriented
spherical fibrations
$$
S^{2r-1}\rightarrow BU(r-1)\rightarrow BU(r).
$$
Let $BU(r) \rightarrow BU(n)$ be the appropriate composite and
define then $E_r$ to be the homotopy pullback of $\widetilde{BN}
\xrightarrow{\rho} BU(n)\pcom \leftarrow BU(r)\pcom$. Thus $E_n =
\widetilde{BN}$ and, for $1 \leq r \leq n$, we have spherical
fibrations $(S^{2r-1})\pcom \rightarrow E_{r-1}\rightarrow E_r$ .
We summarize next some important properties of these spaces.

\begin{proposition}
\label{inductionev}
The component $\map(B\Z/p, E_r)_c$ splits as a product $E_r \times
\map_*(B\Z/p, E_r)_c$ for any $0 \leq r \leq n$.
\end{proposition}

\begin{proof}
Let us denote by $ev_r\colon \map(B\Z/p, E_r)_c\rightarrow E_r$
the evaluation at the component of the constant map and consider
the following commutative square of horizontal fibrations of
connected spaces
\[
\diagram \map(B\Z/p, H)_c \rto \dto^{ev}_{\simeq} &
\map(B\Z/p,E_{r})_c \dto^{ev_{r}} \rto & \map(B\Z/p,
\widetilde{BN})_c \dto^{ev_n} \cr H \rto & E_{r} \rto &
\widetilde{BN}
\enddiagram
\]
The homotopy fiber $H$ has finite cohomology, and is thus
equivalent via the evaluation map to the homotopy fiber
$\map(B\Z/p, H)_c$ of the top map. This shows that the right hand
square is a pull-back square. But $ev_n$ is a trivial fibration by
Proposition~\ref{prop nonconnectedsplitting}. Hence so is $ev_r$
and $\map(B\Z/p, E_r)_c$ must split as a product $E_r \times
\map_*(B\Z/p, E_r)_c$.
\end{proof}

\begin{proposition}
\label{T}
The morphism of $\A_p$-modules $\bar TQH^*(\widetilde{BN}; \F_p) \rightarrow
\bar TQH^*(E_r; \F_p)$ induced by $E_{r}\rightarrow \widetilde{BN}$ is an
isomorphism.
\end{proposition}

\begin{proof}
We have seen in the previous proposition that the map
$E_{r}\rightarrow BX$ induces a weak equivalence on the connected
component of the constant map in the pointed mapping space
$\map_*(B\Z/p, -)_c$. From this point on, the same argument as in
the proof of Theorem~\ref{thm TQforpNoetherian} goes through.
\end{proof}

\begin{proof}[Proof of Theorem~\ref{theorem BX and Ap}]
We know that $H^*(E_0; \F_p)$ is finitely generated as an algebra
over the Steenrod algebra. It is thus sufficient to prove that
$H^*(E_{r}; \F_p)$ is a finitely generated $\A_p$-algebra if so is
$H^*(E_{r-1}; \F_p)$. Denote by $q_r$ the map $E_{r-1} \rightarrow
E_r$ turned into a fibration. Since $\ker q_r^*$ is generated by a
single element, namely the Euler class,
Proposition~\ref{fibration-local-fibre} applies.
\end{proof}

\begin{remark}
\label{rem CCS5}
{\rm In general, an unstable algebra which is finitely generated
as an algebra over $\A_p$ may contain unstable subalgebras which
are not finitely generated over~$\A_p$. Such an example appears in
\cite[Remark~2.2]{CCS5} as an unstable subalgebra $B$ of
$H^*(BS^1\times S^2; \F_p)\cong \F_p[x]\otimes E(y)$, which is not
a finitely generated $B$-module. $B$ is the ideal generated by $y$
turned into an unstable algebra by adding~$1$. The quotient is a
polynomial algebra, in particular it is not finite.}
\end{remark}

\begin{remark}
\label{rem PeterWeyl}
{\rm To prove Theorem~\ref{theorem BX and Ap} one could also use
the fact that $p$-compact groups satisfy the Peter-Weyl theorem
(see \cite[Theorem 1.6]{classificationp} and \cite[Remark
7.3]{classification2}). This would slightly shorten the proof and
avoid the use of the maximal $p$-compact toral subgroup. But the
Peter-Weyl theorem for $p$-compact groups is proved using the
classification of $p$-compact groups and we want to emphasize that
Theorem~\ref{theorem BX and Ap} does not depend on the
classification.}
\end{remark}

\begin{corollary}
\label{cor 4connectedcover}
Let $G$ be a simply connected, simple, compact Lie group. Then
$H^*(BG\langle 4\rangle; \F_p)$ is finitely generated as an algebra
over the Steenrod algebra and $QH^*(BG\langle 4\rangle; \F_p)$
belongs to $\U_1$. In fact there exists a morphism $QH^*(BG\langle
4\rangle; \F_p) \rightarrow F(1) \otimes Q_1$ with finite kernel and
cokernel.
\end{corollary}

\begin{proof}
Recall that $Q_1 = QH^*(B\Z/p; \F_p)$. Since $BG$ is $3$-connected
and $\pi_4 BG \cong \Z$, Proposition~\ref{prop schwartz} yields a a
morphism $QH^*(BG\langle 4\rangle; \F_p) \rightarrow F(1) \otimes
Q_1$ (the elementary abelian group $W$ appearing there has rank
one). The cokernel is finite and the kernel locally finite. But
since $QH^*(BG\langle 4\rangle; \F_p)$ is a finitely generated
module over $\A_p$, the kernel must be finite.
\end{proof}

\begin{example}
\label{more than HK}
{\rm Let $X$ be a either any simply connected compact Lie group
(such as $Spin(10)$, which is one of the smallest examples Harada
and Kono could not handle at the prime $2$) or one of the
$p$-compact groups number $2b$, $23$ or $30$ in the Shephard-Todd
list. For odd primes, the exotic $p$-compact groups arising from
this construction are again non-modular. Hence, in all the cases,
$H^*(X; \Z\pcom)$ is torsion free. The mod $p$ cohomology is given
by $H^*(BX_{23}; \F_p)=\F_p[x_4,x_{12},x_{20}]$, $H^*(BX_{30};
\F_p)=\F_p[x_4,x_{24},x_{40},x_{60}]$, and $H^*(BX_{2b,m};
\F_p)=\F_p[x_4,x_{2m}]$. Since all examples are torsion free, the
same techniques used by Harada and Kono in \cite{MR896947} and
\cite{MR750280} show that $H^*(BX\langle4\rangle; \F_p) \cong
H^*(BX; \F_p)/J\otimes R_h$ where $J$ is the ideal generated by
$x_4,\P1 x_4,\ldots,\P{h} x_4$ for a certain $h$, and $R_h$ is an
unstable subalgebra of $H^*(K(\Z\pcom,3); \F_p)$ which is finitely
generated over~$\A_p$. In fact Theorems~\ref{theorem BX and Ap} and
\ref{thm TQforpNoetherian} show directly that $QH^*(BX \langle 4
\rangle; \F_p)$ is finitely generated as a module over $\A_p$ and
belongs to $\U_1$. From the last corollary we see for example that
$QH^*(BSpin(10)\langle 4\rangle; \F_2)$ differs from $\Sigma F(1)$
in only a finite number of ``low dimensional" classes.}
\end{example}

Our techniques also allow us to say something about the cohomology
of some higher connected covers.

\begin{example}
\label{exotic:allothers}
{\rm Let $X$ be any exotic $p$-compact groups except those
corresponding to number $2b$, $23$ or $30$ in the Shephard-Todd
list. They are $(n-1)$-connected for some integer $n >4$ (for
example at the prime $2$ the only exotic example is $BDI(4)$,
which is $7$-connected). Consider the $n$-connected cover of its
classifying space, $(BX)\langle n \rangle$. Theorem~\ref{theorem
BX and Ap} implies then that $QH^*((BX) \langle n \rangle; \F_p)$
is finitely generated as a module over $\A_p$. It belongs in fact
to $\U_{n-2}$, compare with Corollary~\ref{cor finitecover}. For
example $QH^*((BDI(4))\langle 8 \rangle; \F_2)$ belongs to
$\U_6$.}
\end{example}

\appendix
\section{Modules and algebras} \label{appendix}
Let us consider an unstable algebra $B^*$. Assume there is another
unstable algebra $G^*$ which is finitely generated as an algebra
over~$\A_p$ and which acts on $B^*$. Assume also that $B^*$ is
finitely generated as a module over $G^*$, when can we conclude that
$B^*$ is finitely generated as an algebra over~$\A_p$? If the action
of $G^*$ on $B^*$ is compatible with the action of the Steenrod
algebra, it is obvious, but it is not true in general as illustrated
by the following example. Set $G^* = H^*(K(\Z/2, 2); \F_2)$ and let
$B^*$ be isomorphic to $G^*$ as an algebra, but define the action of
$\A_2$ to be trivial. Then $B^*$ is finitely generated as a
$G^*$-module, but not as an $\A_2$-algebra. We propose a weak notion
of compatibility between the $G^*$-module and $\A_p$-algebra
structures.

\begin{proposition}
\label{algebraprop1}
Let $G^*$ be a connected and commutative finitely generated
$\A_p$-algebra, and let $B^*$ be a connected commutative
$G^*$-algebra which is a finitely generated as a $G^*$-module.
Assume there exists a morphism $p\colon B^*\rightarrow G^*$ of
$G^*$-modules such that $p(\theta(g\cdot 1))=\theta g$ for all
$\theta \in \A_p$ and $g\in G^*$. Then $B^*$ is also a finitely
generated $\A_p$-algebra.
\end{proposition}

\begin{proof}
Since $B^*$ is finitely generated as a $G^*$-module, let us choose
a finite set of generators $b_1=1,\ldots, b_n$, where the degree
of $b_i$ is positive for $i \geq 2$ by the connectivity
assumption. In fact we will assume that $b_i$ belongs to the
kernel of $p$ for $i \geq 2$: replace $b_i$ by $b_i - p(b_i)\cdot
1$ if necessary.

Let $x$ be an element in $B^*$, so $x$ can be written as a
$G^*$-linear combination $x=\sum_{i=1}^n \lambda_i \cdot b_i$. If
$\{z_1,\ldots, z_m\}$ denote the $\A_p$-algebra generators of
$G^*$, then each $\lambda_i$ can be expressed as a polynomial on
Steenrod operations $\theta_i$ applied to the generators,
$\theta_i z_i$ (use the Cartan formula). Our claim is that $b_1,
\dots , b_n$ and $z_1\cdot 1, \dots , z_m \cdot 1$ generate $B^*$
as an algebra over $\A_p$. Since $B^*$ is a unital $G^*$-algebra,
we need to prove that an element of the form $\theta z \cdot 1$,
for $\theta \in \A_p$ and $z \in \{z_1, \dots, z_m \}$, can be
written as a polynomial in the $\theta (z_i \cdot 1)$'s and the
$b_k$'s.

If the action of the Steenrod algebra were compatible with the
module action, one would have that $\theta z \cdot 1 = \theta (z
\cdot 1)$. This is not the case, but it is sufficient to deal with
the element $\xi= \theta z \cdot 1 - \theta (z \cdot 1)$. Note
that $\xi \in \ker p$ since $p(\theta z \cdot 1)=\theta z$. We
will proceed by induction on the degree. If $\xi$ is in degree
zero, then the statement is clear since the algebras are
connected. Assume that the statement is true for degrees $<|\xi|$
and write $\xi = \lambda_1 \cdot 1 + \lambda_2 \cdot b_2 + \dots
\lambda_n \cdot b_n$. By the induction hypothesis we know that the
elements $\lambda_i \cdot 1$ can be expressed as polynomials in
the $\theta (z_i \cdot 1)$'s and the $b_k$'s. So we only need to
deal with $\lambda_1 \cdot 1$. But $\lambda_1 = p(\lambda_1 \cdot
1) = p(\xi - \lambda_2 \cdot b_2 - \dots - \lambda_n \cdot b_n)$
which is zero because $p$ is a morphism of $G^*$-modules and $b_i
\in \ker p$ for $i \geq 2$. This concludes the proof.
\end{proof}

In Section~\ref{sec Serre} we need a slight generalization of this
proposition.

\begin{proposition}
\label{algebraprop}
Let $G^*$ be a connected and commutative finitely generated
$\A_p$-algebra, $C^*$ be a connected and finitely generated
algebra, and let $B^*$ be a connected commutative $G^* \otimes
C^*$-algebra which is a finitely generated as a $G^* \otimes
C^*$-module. Assume there exists a morphism $p\colon
B^*\rightarrow G^*$ of $G^*$-modules such that $p(\theta(g\cdot
1))=\theta g$ for all $\theta \in \A_p$ and $g\in G^*$. Then $B^*$
is also a finitely generated $\A_p$-algebra.
\end{proposition}

\begin{proof}
Just as in the previous proof, the only problem is to write an
element of the form $(\theta z) \cdot 1$, with $\theta \in \A_p$
and $z \in G^*$, in terms of the generators $z_i$, $b_k$, and
generators $c_m$ of the algebra $C^*$. The proof is then basically
the same.
\end{proof}

\bibliographystyle{amsplain}
\providecommand{\bysame}{\leavevmode\hbox
to3em{\hrulefill}\thinspace}
\providecommand{\MR}{\relax\ifhmode\unskip\space\fi MR }
\providecommand{\MRhref}[2]{%
  \href{http://www.ams.org/mathscinet-getitem?mr=#1}{#2}
} \providecommand{\href}[2]{#2}



\bigskip
{\small
\begin{center}
\begin{minipage}[t]{8 cm}
Nat\`{a}lia Castellana and Jer\^{o}me Scherer\\ Departament de Matem\`atiques,\\
Universitat Aut\`onoma de Barcelona,\\ E-08193 Bellaterra, Spain
\\\textit{E-mail:}\texttt{\,Natalia@mat.uab.es}, \\
\phantom{\textit{E-mail:}}\texttt{\,\,jscherer@mat.uab.es}
\end{minipage}
\begin{minipage}[t]{6 cm}
Juan A. Crespo \\ Departamento de Econom\'{\i}a Cuantitativa,
\\ Universidad Aut\'onoma de Madrid,\\ E-28049, Cantoblanco Madrid,
Spain
\\\textit{E-mail:}\texttt{\,juan.crespo@uam.es},
\end{minipage}
\end{center}}

\end{document}